\documentclass{amsart}
\usepackage{articletemplate}
\numberwithin{theorem}{section}

\begin{document}
\title{Moment estimates for the exponential sum with higher divisor functions}
\author{Mayank Pandey}
\address{Department of Mathematics, California Institute of Technology, Pasadena, CA 91125}
\email{mpandey@caltech.edu}
\maketitle

\section{Introduction}

For a sequence $(a_n)_{n\ge 1}$ of arithmetic interest, it is often desirable to have estimates for the $L^p$ norms of the exponential sum $M(\alpha) = \sum_{n\le X}a_ne(n\alpha)$ as $X$ grows. Such estimates are
useful in applications of the circle method. In addition, sufficiently strong estimates for them can yield estimates for the distribution function 
$\{\alpha\in [0, 1] : |M(\alpha)|\ge \lambda\}$ for $\lambda$ in appropriate ranges. 

In the case that $a_n$ is $1$ if $n$ is a $k$th power and $0$ otherwise, such estimates have connections to Waring's problem, and the consequences of conjectured estimates for $\int_0^1 |M(\alpha)|^sd\alpha$
for $s$ in various ranges have been studied by Vaughan and Wooley \cite{VW}. 

This problem was also studied by Keil \cite{K} in the case of the indicator function of $k$-free numbers, and the size of $\int_0^1 |M(\alpha)|^sd\alpha$ was estimated up to a constant 
factor for all $s\ne 1 + \frac{1}{k}$, and in the case $s = 1 + \frac{1}{k}$, it was only determined up to a factor of $\log X$. 

In general, when higher values of $s$ are considered, as long as the sequence in question has some structure in arithmetic progressions, the bulk of the contribution ends up coming from narrow regions near a 
small number of points (typically rationals with small denominator). For this reason, one typically expects that $\int_0^1 |M(\alpha)|^sd\alpha$ is between 
$X^{-\eps}A_s(X)$ and $X^{\eps}A_s(X)$ with $A_s(X)$ equal to either $X^{\alpha_1s}$ or $X^{\alpha_1s} + X^{\alpha_2s - \sigma_1}$ for some $\alpha_1 < \alpha_2$, and some $\sigma_1 > 0$. 
The second case is what happens in the case of $k$-free numbers, as shown in \cite{K}, as well what is conjectured in the case of $k$th powers (see \cite{VW} for more discussion of this). 
In the case of the M\"obius function, the first case is conjectured (it is implied by Mertens conjecture that $|M(\alpha)|\ll X^{1/2 + \eps}$). 

In this paper, we study the case of divisor functions and high moments. In particular, let $k\ge 2$ be some integer, and $s > 2$ be real. Then, let 
\[\tau_k(n) = \sum_{d_1\dots d_k = n} 1\]
and 
\[M(\alpha) = \sum_{n\le X} \tau_k(n)e(n\alpha).\]
Our main result is the following.
\begin{theorem}
We have 
\[\int_0^1 |M(\alpha)|^sd\alpha = X^{s -  1}(\log X)^{s(k - 1)}\sum_{\ell\ge 0}\frac{\gamma_{\ell, s, k}}{(\log X)^\ell} 
+ O(X^{s - 1 - \delta_{s, k} + \eps})\]
with
\begin{align*}
\delta_{s, k} = \frac{2(s - 2)}{(s + 7)(k + 1) + 2}.
\end{align*}
for some coefficients $\gamma_{s, k,\ell}$ satisfying the bound $|\gamma_{s, k,\ell}|\ll \exp(O(\ell))$, with $\gamma_{s, k, 0} > 0$. 
\label{thm:main_result}
\end{theorem} 

We prove this with a straightforward application of the circle method. For such high moments, the contribution near rationals with small denominator 
(the major arcs) dominates. We therefore require bounds for the remaining points (the minor arcs).

The minor arc bounds we use follow from a decomposition of $\tau_k$ into type I and type II sums. Vaughan's identity in the proof of analogous bounds for the von Mangoldt function provides this decomposition, though 
the convolution structure of $\tau_k$ makes the decomposition somewhat more straightforward.
The major arc estimates follow from standard estimates for partial sums of 
$\tau_k(n)\chi(n)$ coming from Voronoi summation (in particular, Theorem 4.16 in \cite{IK}).

In the course of dealing with the main term, we use a result on the order of magnitude of higher moments of Dirichlet kernels, which we state here. A proof of this will be given in a later section. Here, we write 
\[v(\beta) = \sum_{n\le X} e(n\beta).\]
\begin{proposition}
We have that for $s > 2$
\[\int_0^1 |v(\beta)|^s d\beta = A_sX^{s - 1} + O\left(X^{s - 2}\right).\]
where
\[A_s = \frac{2}{\pi}\int_0^\infty \frac{|\sin t|^s}{t^s}dt.\]
\label{prop:dir_ker_moment}
\end{proposition}

Our methods likely generalize straightforwardly to the case of $\chi_1 * \dots *\chi_k$ for some fixed Dirichlet characters $\chi_1,\dots,\chi_k$, and yield a similar result. The case of Fourier coefficients of $\GL(k)$ 
cusp forms is quite distinct however, since it is expected, and was shown by Jutila \cite{J} for some of the $\GL(2)$ case, that the relevant exponential sum is small everywhere. Consequently, the bulk of the contribution
should not be expected to come from the major arcs, so the method used here fails.

We have not taken much care to optimize the sizes of the error terms. In particular, the error terms in Proposition \ref{prop:maj_arc1} can likely be improved quite cheaply. However, an error term
qualitatively superior to $\delta_{s,k}\ll_s\frac 1k$ is likely quite hard to breach.

\subsection{Notation and conventions}
$X$ is some sufficiently large real number that should be thought of as going to $\infty$, and $\eps > 0$ is some sufficiently small constant. $s > 2$ is a fixed real number, and $k\ge 2$ is some fixed integer.
As usual, we use the notation $A\ll B\iff A\le O(B)\iff B\gg A$ to denote that $|A|\le CB$ for some absolute constant $C$. In any instance, this implied constant may depend on $s, k, \eps$, and any further parameters 
on which it may depend will be noted in a subscript. We write $a\sim A$ to denote that $A < a\le 2A$, and $a\asymp A$ to denote that $A\ll a\ll A$. 
\section{Setup}

Take $P = X^\eta$, with $\eta = \frac{2}{s - 2}\delta_{s, k}$.
It is easy to see that we have the bound $\eta\le\frac{2}{5}$.
$\eta$ also has the property that $\frac{2}{k + 1} - \bigg(\frac{9}{2} + \frac{1}{k + 1}\bigg)\eta = \frac{1}{2}\eta(s - 2)$. The significance of this will become clear later on when we are collecting
various error terms.
Also, let $\mf M$ be the union of 
\[ \mf M(q, a) = \{\alpha\in [0, 1] : |\alpha - a/q|\le PX^{-1}\}\] 
for $q\le P, (a, q) = 1$, and $\mf m = [0, 1]\setminus\mf M$. Note that for large $X$, all the $\mf M(q, a)$ are disjoint.
It is easy to see by Dirichlet's approximation theorem that for all $\alpha\in\mf m$, there exist $P < q\le X/P, (a, q) = 1$ so that 
$|\alpha - a/q|\le q^{-2}$. Then, the main result follows if we can prove the following two estimates for the contribution of the major and minor arcs.
\begin{proposition}
We have that 
\[\int_{\mf M} |M(\alpha)|^sd\alpha = X^{s -  1}(\log X)^{s(k - 1)}\sum_{\ell\ge 0}\frac{\gamma_{\ell, s, k}}{(\log X)^\ell} 
+ O(X^{s - 1 - \delta_{s, k} + \eps})\]
where $\gamma_{s, k, \ell},\delta_{s, k}$ are as in the statement of Theorem \ref{thm:main_result}.
\label{prop:MT_est}
\end{proposition}
\begin{proposition}
We have the bound
\[\int_{\mf m} |M(\alpha)|^sd\alpha\ll X^{s - 1 - \delta_{s, k}}(\log X)^{O(1)}.\]
\label{prop:min_arc_bound}
\end{proposition}
\begin{proof}
This follows immediately from Proposition \ref{prop:min_arc_est}, whose proof we defer to the last section, and Parseval. Indeed, note that since $\eta\le 2/5$, it follows from Proposition \ref{prop:min_arc_est}
that $\sup_{\alpha\in\mf m} |M(\alpha)|\ll X^{1 - \eta / 2}(\log X)^{O(1)}$, and therefore
\begin{align*}
\int_{\mf m} |M(\alpha)|^sd\alpha&\ll \bigg(\sup_{\alpha\in\mf m} |M(\alpha)|\bigg)^{s - 2}\int_0^1 |M(\alpha)|^2d\alpha\\ 
&\ll (X^{1 - \eta/2})^{s - 2}X(\log X)^{O(1)}\ll X^{s - 1 - \frac{1}{2}\eta(s - 2)}(\log X)^{O(1)}.
\end{align*}
The proposition follows upon noting that $\frac{1}{2}\eta(s - 2) = \delta_{s, k}$.
\end{proof}

In the next section, we shall prove Propositions \ref{prop:MT_est}, \ref{prop:min_arc_est}.
The main theorem clearly follows from Propositions \ref{prop:MT_est} and \ref{prop:min_arc_bound}.

\section{Major arc estimates for higher divisor functions}
Our main major arc estimate is the following. 
\begin{proposition}
Suppose that $q\ge 1, (a, q) = 1$. Then, we have that
\[\sum_{n\le X} \tau_k(n)e\left(\frac{an}{q}\right) = XP_{k, q}(\log X) + O(q^{\frac{1}{2} + \frac{k}{k + 1}}X^{\frac{k - 1}{k + 1}}(qX)^{\eps})\]
where $P_{k, q}(\log X)$ is a polynomial of degree $k - 1$ in $\log X$ with coefficients of size $\ll \tau_2(q)^{O(1)}/q$. In addition, the coefficient of $(\log X)^{k - 1}$ is nonnegative and $\gg 1/q$.
\label{prop:maj_arc1}
\end{proposition}
\begin{proof}
This follows from the method in the proof of Proposition 4.2 in \cite{MRT}, though we may use Theorem 4.16 in \cite{IK} to achieve the above error terms.
\end{proof}
From partial summation, we then obtain the following.
\begin{corollary}
Suppose that $q\ge 1, (a, q) = 1, |\beta|\le 1$. Then, we have that 
\[\sum_{n\le X}\tau_k(n)e\left(\frac{an}{q} + n\beta\right) = Q_{k, q}(\log X)v(\beta) + O((1 + |\beta|X)q^{\frac{1}{2} + \frac{k}{k + 1}}X^{\frac{k - 1}{k + 1}}(qX)^{\eps})\]
\label{cor:maj_arc_est}
where $Q_{k, q}(\log X)$ is a polynomial of degree $k - 1$ in $\log X$ with coefficients of size $\ll \tau_2(q)^{O(1)}/q$. In addition, the leading coefficient is nonnegative and $\gg 1/q$.
\end{corollary}

Before we start dealing with the main term, we shall prove Proposition \ref{prop:dir_ker_moment}.
\begin{proof}[Proof of Proposition \ref{prop:dir_ker_moment}]
Our proof here uses the method in a MathStackExchange post of daniel-fischer \cite{DF}, though we take some care here to track the error terms.

We shall suppose for simplicity that $X$ is an integer, as it can be easily checked that adjusting $X$ by $O(1)$ does not alter the main term on the RHS by a quantity that can't be absorbed into the 
error term. 

It is well-known then that $v(\beta) = \frac{\sin(\pi (X + 1)\beta)}{\sin(\pi\beta)}$. Taylor expanding $\frac{\pi\beta}{\sin\pi\beta}$, we have that for $\beta\in [0, 1/2]$
\[\frac{\pi\beta}{\sin\pi\beta} = 1 + O(\beta^2)\]
and it can be easily checked that $|\pi\beta/\sin(\pi\beta) - 1|\le 3/4$. Therefore, we have that
\[\bigg(1 + \bigg(\frac{\pi\beta}{\sin\pi\beta} - 1\bigg)\bigg)^p = 1 + O(\beta^2)\]
so
\[\int_0^1 |v(\beta)|^sd\beta = 2\int_0^{1/2} |\sin((X + 1)\pi\beta)|^s(\pi\beta)^{-s} d\beta + O\bigg(\int_0^{1/2}|\sin((X + 1)\pi\beta)|^s\beta^{2 - s}d\beta\bigg).\]
By the bound $|\sin((X + 1)\pi\beta)|^s\ll \min(1, (\beta X)^s)$ that the term inside the $O(-)$ is 
\[\ll\int_0^{1/X} (\beta X)^s\beta^{2 - s}d\beta + \int_{1/X}^{1/2} \beta^{2 - s}d\beta\ll X^{s - 2}.\]
Now, by a change of variables, the main term equals 
\begin{align*}
\frac{2}{\pi(X + 1)}\int_0^{(X + 1)\pi/2} |\sin t|^s(t/(X + 1))^{-s} dt &= \frac{2}{\pi}(X + 1)^{s - 1}\int_0^{(X + 1)\pi/2} \frac{|\sin t|^s}{t^s}dt\\
        &= \frac{2}{\pi}X^{s - 1}\int_0^\infty \frac{|\sin t|^s}{t^s}dt + O(X^{s - 2}).
\end{align*}
as we have by a trivial bound that
\[\int_0^{(X + 1)\pi/2} \frac{|\sin t|^s}{t^s}dt = \int_0^\infty \frac{|\sin t|^s}{t^s}dt + O(X^{1 - s}).\]
The desired result follows.
\end{proof}
We will now prove Proposition \ref{prop:MT_est} using Proposition \ref{prop:dir_ker_moment}.
\begin{proof}[Proof of Proposition \ref{prop:MT_est}]
From the definition of $\mf M$, we have that
\begin{align*}
\int_{\mf M} |M(\alpha)|^sd\alpha &= \sum_{q\le P}\sumCp_{a(q)}\int_{-P/X}^{P/X} \bigg|M\left(\frac{a}{q} + \beta\right)\bigg|^s d\beta\\
        &= \sum_{q\le P}\varphi(q)Q_{k, q}(\log X)^s\int_{-P/X}^{P/X} |v(\beta)|^s d\beta + O(X^{s - 1} P^{\frac{9}{2} + \frac{1}{k + 1}}X^{-\frac{2}{k + 1} + \eps}).
\end{align*}
We may extend the range of integration to $[-1/2, 1/2]$ at a total loss of $\ll P(X/P)^{s - 1} (\log X)^{O(1)}\ll X^{s - 1 - (s - 2)\eta}(\log X)^{O(1)}$ by the bound $v(\beta)\ll \min(X, \norm\beta^{-1})$. 
Applying Proposition \ref{prop:dir_ker_moment} then yields that the above equals
\begin{align*}
A_sX^{s - 1}\sum_{q\le P}\varphi(q)Q_{k, q}(\log X)^s  + O(X^{s - 1} P^{\frac{9}{2} + \frac{1}{k + 1}}X^{-\frac{2}{k + 1} + \eps} + X^{s - 1}X^{-(s - 2)\eta + \eps}).
\end{align*}
Now, writing $Q_{k, q}(\log X) = \alpha_0(q) + \dots + \alpha_{k - 1}(q)(\log X)^{k - 1}$, we obtain that 
\begin{align*}
Q_{k, q}(\log X)^s &= (\log X)^{s(k - 1)}\alpha_0(q)^s\left(1 + \frac{\alpha_1(q)\alpha_0(q)^{-1}}{\log X} + \dots + \frac{\alpha_{k - 1}(q)\alpha_0(q)^{-1}}{(\log X)^{k - 1}}\right)^s \\
        &= \alpha_0(q)^{s}(\log X)^{s(k - 1)}\sum_{\ell\ge 0}\frac{(s)\dots (s - \ell + 1)}{\ell!}\cdot\frac{\beta_\ell(q)}{(\log X)^\ell}
\end{align*}
for some coefficients $\beta_{\ell}(q)$ with $\beta_0(q) = 1$, and $|\beta_{\ell}(q)|\ll\tau_2(q)^{O(\ell)}$ for $\ell\ge 1$. Here, we have use the fact that $\alpha_0(q)$ is nonnegative and $\gg 1/q$.
Executing the summation over $q$, we thus obtain that 
\[\int_{\mf M} |M(\alpha)|^sd\alpha = X^{s -  1}(\log X)^{s(k - 1)}\sum_{\ell\ge 0}\frac{\gamma_{s, k,\ell}}{(\log X)^\ell} 
+ O(X^{s - 1 + \left(\frac{9}{2} + \frac{1}{k + 1}\right)\eta - \frac{2}{k + 1}  + \eps} + X^{s - 1 - (s - 2)\eta + \eps})\]
for some coefficients $\gamma_{s, k, \ell}$ satisfying the bound $|\gamma_{s, k,\ell}|\ll \exp(O(\ell))$. The desired result follows from our choice of $\eta$, as $(s - 2)\eta = 2\delta_{s, k}$, and as we noted 
previously
\[\left(\frac{9}{2} + \frac{1}{k + 1}\right)\eta - \frac{2}{k + 1} = -\frac{s - 2}{2}\eta = -\delta_{s, k}.\]
\end{proof}

\section{The minor arcs}
To bound $M(\alpha)$ on the minor arcs, we shall use the following bound. This is essentially the same bound one obtains in the case of the von Mangoldt function. Our proof proceeds in the same manner as
this case, through a decomposition of $\tau_k(n)$ into type I and type II sums.

\begin{proposition}
Supposed that $\alpha, a, q$ are so that $(a, q) = 1, |\alpha - \frac{a}{q}|\le 1/q^2$. Then, we have that 
\[\bigg|\sum_{n\le X} \tau_k(n)e(n\alpha)\bigg|\ll \left(\sqrt{qX} + \frac{X}{\sqrt{q}} + X^{4/5}\right)(\log X)^{O(1)}\]
\label{prop:min_arc_est}
\end{proposition}
\begin{proof}
First, it is easy to see by splitting into dyadic intervals that it suffices to show the result with a sum over $n\sim X$, so we shall assume this from now on.

Our proof follows similarly to the proof of minor arc bounds for the exponential sum with the von Mangoldt function, though our decomposition into type I and II sums will follow straightforwardly from the 
structure of $\tau_k$ as a Dirichlet convolution. Note that 
\[\tau_k\charf{[1, 2X]} = \bigg(\substack{\underbrace{\charf{[1, 2X]}*\dots*\charf{[1, 2X]}}\\\text{k times}}\bigg)\charf{[X, 2X]}. \]
Decomposing $[1, 2X]$ into dyadic intervals then yields that this is a linear combination (with coefficients of size $O(1)$) of $O((\log X)^{O(1)})$ summands of the form
\[(\charf{I_1}*\dots*\charf{I_k})\charf{[X, 2X]}\]
where $I_j$ is of the form either $[N_j, 2N_j]$ or $[N_j, 2X]$ (with $N_j\ge X$ in the second case) for all $j$, for some $1\le N_1\le\dots\le N_k$
satisfying $N_1\dots N_k\asymp X$. It suffices then to show the bound in the proposition for sums of the form
\[\sum_{n\sim X} (\charf{I_1}*\dots*\charf{I_k})(n)e(n\alpha).\]
We have two cases. If all the $N_k\le X^{1/5}$, then there exists a $j$ so that $X^{2/5}\ll N_1\dots N_j, N_{j + 1}\dots N_k\ll X^{3/5}$
so it follows that the sum equals 
\[\sum_{\substack{m\asymp N_1\dots N_j\\ n\asymp N_{j + 1}\dots N_k}} a(m)b(n)e(\alpha mn)\]
for some coefficients $a(m), b(n)$ so that $|a(m)|\ll \tau_j(m), |b(n)|\ll\tau_{k - j}(n)$. The bound then follows from a standard bound for type II sums (Lemma 13.8 in \cite{IK}, for example, suffices).

Otherwise, we have that $N_k > X^{1/5}$, so the sum equals 
\[\sum_{n\asymp N_k, m\asymp  N_1\dots N_{k - 1}} a(m)e(\alpha mn)\] 
for some coefficients $a(m)$ bounded by $\tau_{k - 1}(n)$, so the desired result then follows standard bounds on type I sums, such as Lemma 13.7 in \cite{IK}.

\end{proof}

\emph{Acknowledgements.} The author would like to thank the anonymous referee for various corrections as well as for pointing out the simplified decomposition into type I and II sums used in the 
proof of Proposition \ref{prop:min_arc_est}.

\end{document}